\documentclass[11pt]{amsart}
\usepackage{fixltx2e}
\usepackage[utf8x]{inputenc}
\usepackage[T1]{fontenc}
\usepackage[paper=a4paper]{geometry} 
\usepackage[light, frenchstyle,narrowiints,partialup,sfmathbb]{kpfonts}
\usepackage{mathtools}
\usepackage{hyperref}
\usepackage[english]{babel}
\usepackage{microtype}
\hypersetup{%
colorlinks=true,%
linkcolor=black,%
anchorcolor=black,%
citecolor=black,%
menucolor=black,%
urlcolor=black,%
baseurl={http://carva.org/emmanuel.royer},%
pdftitle={Sign changes in short intervals of coefficients of spinor zeta function of a Siegel cusp form of genus 2},%
pdfauthor={Emmanuel Royer (Universit\'e Blaise Pascal, CNRS), Jyoti Sengupta (TIFR), Jie Wu (CNRS, Universit\'e de Lorraine)},%
pdfsubject={spinor zeta function, Siegel form, Fourier coefficients},%
pdfkeywords={spinor zeta function, Siegel form, Fourier coefficients, Cesaro convergence, Voronoi formula},%
unicode=true,%
pdflang=en,
}
\newtheoremstyle{erthm}
  {}
  {}
  {\itshape}
  {}
  {\fontseries{bx}\selectfont\itshape}
  {--}
  { }
  {}
  \newtheoremstyle{errem}
    {}
    {}
    {}
    {}
    {\ttfamily\itshape}
    {--}
    { }
    {}
\theoremstyle{erthm}
\newtheorem*{theorem}{Theorem}
\newtheorem{lemma}{Lemma}
\theoremstyle{errem}
\newtheorem*{remark}{Remark}
\title[Sign changes of coefficients of spinor zeta function]{Sign changes in short intervals of coefficients of spinor zeta function of a Siegel cusp form of genus 2}
\thanks{The work was supported by a grant from the Indo-French Centre for the Promotion of Advanced Research (CEFIPRA Project No. 4601-2). Part of this work has been done during the visit of the second author at l'Institut \'Elie Cartan de Lorraine and finished during the visit of the first and third authors at Tata Institute of Fundamental Research. They would like to thank these institutes for hospitality. }
\author[E. Royer]{Emmanuel Royer}
\address{%
Emmanuel Royer\\
Clermont Universit\'e\\
Universit\'e Blaise Pascal\\
Laboratoire de math\'ematiques\\
BP 10448\\
F-63000 Clermont-Ferrand\\
France %
}
\curraddr{%
Emmanuel Royer\\
Universit\'e Blaise Pascal\\
Laboratoire de math\'ematiques\\
Les C\'ezeaux\\
BP 80026\\
F-63171 Aubi\`ere Cedex\\
France %
}
\email{{emmanuel.royer@math.univ-bpclermont.fr}}
\author[J. Sengupta]{Jyoti Sengupta}
\address{Jyoti Sengupta\\
School of Mathematics \\
T.I.F.R. \\
Homi Bhabha Road \\
400\ 005 Mumbai\\
India}
\email{sengupta@math.tifr.res.in}
\author[J. Wu]{Jie Wu}
\address{%
Jie Wu\\
CNRS\\
Institut \'Elie Cartan de Lorraine\\
UMR 7502\\
F-54506 Van\-d\oe uvre-l\`es-Nancy\\
France}
\curraddr{%
Université de Lorraine\\
Institut \'Elie Cartan de Lorraine\\
UMR 7502\\
F-54506 Van\-d\oe uvre-l\`es-Nancy\\
France
}
\email{jie.wu@univ-lorraine.fr}
\DeclarePairedDelimiter\abs{\lvert}{\rvert}
\newcommand*{\CC}{\mathbb{C}}
\newcommand*{\dd}{%
  \mathop{\mathrm{d}\null}\mskip-\thinmuskip\mathord{\null}}
\newcommand*{\e}{\mathrm{e}} 
\DeclarePairedDelimiter\ent{\lfloor}{\rfloor}
\renewcommand*{\epsilon}{\varepsilon} 
\renewcommand*{\geq}{\geqslant} 
\newcommand*{\GL}{\mathrm{GL}}
\newcommand*{\ic}{\mathrm{i}} 
\renewcommand*{\leq}{\leqslant} 
\newcommand*{\N}{\mathbb{N}}
\newcommand*{\prem}{\mathcal{P}}
\newcommand*{\Q}{\mathbb{Q}}
\newcommand*{\re}{\mathord{R\mkern-1mu e}} 
\DeclareMathOperator{\sgn}{sgn} 
\newcommand*{\Sp}{\mathrm{Sp_4}}
\newcommand*{\Z}{\mathbb{Z}}
\begin{document}
\mathtoolsset{showonlyrefs,mathic,centercolon}
\begin{abstract}
In this paper, we establish a Voronoi formula for the spinor zeta function of a Siegel cusp form of genus \(2\). We deduce from this formula quantitative results on the number of its positive (resp. negative) coefficients in some short intervals. %
\end{abstract}
\maketitle
\tableofcontents%
\section{Introduction}

Let \(S_k\) be the space of Siegel cusp forms of integral weight \(k\) on the group \(\Sp(\Z)\subset \GL_4(\Q)\) and let \(F\in S_k\) be an eigenfunction of all the Hecke operators. 
Let
\begin{equation}\label{defZFsEuler}
Z_F(s):=\prod_{p\in\prem} Z_{F, p}(p^{-s})\quad (\re s> 1)
\end{equation}
be the spinor zeta function of \(F\). Here \(\prem\) is the set of prime numbers and if \(\alpha_{0, p},\alpha_{1, p},\alpha_{2, p}\) are the Satake \(p\)-parameters attached to \(F\) then
\begin{equation}\label{defZFpt}  
Z_{F, p}(t)^{-1}:=(1-\alpha_{0, p}t)(1-\alpha_{0, p}\alpha_{1, p}t)(1-\alpha_{0, p}\alpha_{2, p}t)(1-\alpha_{0, p}\alpha_{1, p}\alpha_{2, p}t). %
\end{equation}
They satisfy %
\begin{equation}\label{productalphajp}
\alpha_{0, p}^2 \alpha_{1, p}\alpha_{2, p} = 1
\end{equation}
for all \(p\).  A Siegel form is in the Maass subspace \(S_k^M\) of \(S_k\) if it is a linear combination of Siegel forms \(F\) that are eigenvectors of all the Hecke operators and for which there exists a primitive modular form, \(f\), of weight \(2k-2\) such that %
\[ %
Z_F(s)=\zeta\left(s-\frac{1}{2}\right)\zeta\left(s+\frac{1}{2}\right)L(f,s). %
\]
Here \(L(f,s)\) is the \(L\)-function of \(f\) (note that we normalise all the \(L\)-functions so that the critical strip is \(0\leq\re s\leq 1\) and the functional equation relates the value at \(s\) to the value at \(1-s\)). This happens only if \(k\) is even. The bijective linear application between \(S_k^M\) and the space of modular forms of weight \(2k-2\) is called the Saito-Kurokawa lifting~\cite{MR633910}. The Ramanujan-Petersson conjecture says that
\begin{equation}\label{RPConjecture}
\abs{\alpha_{j, p}}=1 \text{ for \(j=0, 1, 2\) and all primes \(p\).}
\end{equation}
It is not true for Siegel Hecke-eigenforms in \(S_k^M\). But, if \(k\) is odd or, if \(k\) is even and the form is in the orthogonal complement of \(S_k^M\), then it has been established by Weissauer~\cite{MR2498783}. We denote by \(H_k^*\) the set of Siegel cuspidal Hecke-eigenforms of weight \(k\) and genus \(2\) that, if \(k\) is even, are in the orthogonal complement of \(S_k^M\). The forms we consider in this paper all belong to \(H_k^*\). According to Breulmann~\cite{MR1719682}, a Siegel Hecke-eigenform is in \(S_k^M\) if and only if all its Hecke eigenvalues are positive.

According to \cite{MR0432552, MR575936}, the function
\begin{equation}\label{defLambdaFs}
\Lambda_F(s):= (2\pi)^{-s} \Gamma\left(s+k-\tfrac{3}{2}\right)\Gamma\left(s+\tfrac{1}{2}\right) Z_F(s)
\end{equation}
has an entire continuation to \(\CC\) since \(F\in H_k^*\). Further it satisfies the functional equation
\begin{equation}\label{FE}
\Lambda_F(s)=(-1)^k \Lambda_F(1-s)
\end{equation}
on \(\CC\). The spinor zeta function of \(F\) has the Dirichlet expansion:
\begin{equation}\label{defZFs}
Z_F(s) = \sum_{n\geq 1} a_F(n) n^{-s}
\end{equation}
for \(\re s>1\). By using \eqref{RPConjecture}, one sees that
\begin{equation}\label{RP-Conjecture}
\abs*{a_F(n)}\leq d_4(n)
\end{equation}
for all \(n\geq 1\), where \(d_4(n)\) is the number of solutions in positive integers \(a,b,c,d\) of \(n=abcd\).

In this paper, we investigate the problem of sign changes for the sequence \(\left(a_F(n)\right)_{n\geq 1}\) in short intervals. Define
\begin{equation}
\mathscr{N}_{F}^{+}(x):= \sum_{\substack{n\leq x\\ a_F(n)>0}} 1\quad\text{and}\quad \mathscr{N}_{F}^{-}(x):= \sum_{\substack{n\leq x\\ a_F(n)<0}} 1. %
\label{defNFx}
\end{equation}
We apply a method due to Lau \& Tsang \cite{MR1887887} to establish the following Theorem. Convergence issues however appear and we have to deal with them.

\begin{theorem} 
Let \(F\) be in \(H_k^*\) and \(\varepsilon>0\). There are constants \(c>0\) absolute and \(x_0(F)\) depending only on \(F\) such that for all \(x\geq x_0(F)\), we have
\begin{equation}
\mathscr{N}_F^{+}(x+cx^{3/4})-\mathscr{N}_F^{+}(x)\gg x^{3/8-\varepsilon},
\end{equation}
and %
\begin{equation}
\mathscr{N}_F^{-}(x+cx^{3/4})-\mathscr{N}_F^{-}(x)\gg x^{3/8-\varepsilon},
\end{equation}
where the implied constants in \(\gg\) depends only on \(\varepsilon\). 
\end{theorem}
\begin{remark}
An ingredient of our proof is the inequality %
\begin{equation}\label{MeanValue1/2}
\sum_{n\leq x} a_F(n) \ll_{F, \varepsilon} x^{3/5+\varepsilon}\quad (x\geq 2).
\end{equation}
(see Lemma~\ref{lemmaVoronoi}). We also prove, and use an Omega-result: %
\[ %
\sum_{n\leq x} a_F(n)=\Omega_{\pm}(x^{3/8})
\]
(see Lemma~\ref{lem1}).
\end{remark}

Two related problems have already been studied. Denote by \(\lambda_F(n)\) the \(n\)-th normalised Hecke eigenvalue of \(F\). Then we have
\begin{equation}\label{Relationlambdaa}
\sum_{n=1}^\infty \frac{\lambda_F(n)}{n^s} = \frac{Z_F(s)}{\zeta(2s+1)}\quad (\re s>1).
\end{equation}
In \cite{MR2262899}, Kohnen proved that 
\begin{equation}
\#\{n\leq x\colon \lambda_F(n)>0\}\to\infty\qquad (x\to\infty).
\end{equation}
and %
\begin{equation}
\#\{n\leq x\colon \lambda_F(n)<0\}\to\infty\qquad (x\to\infty).
\end{equation}

Then, Das ~\cite{DAS_IJNT_2013} proved that, as \(x\) tends to \(+\infty\), the quantities 
\begin{equation}
\frac{1}{\#\{p\in\prem\colon p\leq x\}}
\#\{p\in\prem\cap[1,x]\colon \lambda_F(p)>0\} %
\end{equation}
and
\begin{equation}
\frac{1}{\#\{p\in\prem\colon p\leq x\}}
\#\{p\in\prem\cap[1,x]\colon \lambda_F(p)<0\} %
\end{equation}
are bounded from below (and naturally also bounded from above). %
In \cite{MR2326486}, Kohnen \& Sengupta proved that under the same assumption there is an integer \(n\ll k^2(\log k)^{20}\) such that \(\lambda_F(n)<0\). Their result has been generalised to higher levels by Brown~\cite{MR2661286}. An interesting study of sign changes is also due to Pitale \& Schmidt~\cite{MR2425722}. They prove that if \(F\) is not in the Maass subspace, there exists an infinite set of prime numbers \(p\) not dividing the level so that there are infinitely many \(r\) with \(\lambda_F(p^r)>0\) and infinitely many \(r\) with \(\lambda_F(p^r)<0\). %
\begin{remark}
Das' result is on the counting function of the Hecke eigenvalues. It implies that, as \(x\) tends to \(+\infty\), the quantities 
\begin{equation}
\frac{1}{\#\{p\in\prem\colon p\leq x\}}
\#\{p\in\prem\cap[1,x]\colon a_F(p)>0\} %
\end{equation}
and
\begin{equation}
\frac{1}{\#\{p\in\prem\colon p\leq x\}}
\#\{p\in\prem\cap[1,x]\colon a_F(p)<0\} %
\end{equation}
are bounded from below. The reason is that~\eqref{Relationlambdaa} implies %
\[ %
a_F(n)=\sum_{\substack{(d,m)\in\N^2\\ d^2m = n}}\frac{\lambda_F(m)}{d}. %
\]
Thus \(a_F(n) = \lambda_F(n)\) for \(n\) squarefree and in particular for \(n\) a prime. Moreover, the proof of Kohnen \& Sengupta can be adapted to prove that there is an integer \(n\ll k^2(\log k)^{20}\) such that \(a_F(n)<0\).
\end{remark}

To end this introduction, we give a very short amount on what is known in the case of classical modular forms, referring to~\cite{Lau_Liu_Wu} for a more complete survey. Let \(f\) be a primitive modular form of weight \(k\) on the congruence subgroup \(\Gamma_0(N)\). Lau \& Wu~\cite{MR2551607} proved that, as \(x\) tends to \(+\infty\), the quantities 
\begin{equation}\label{eq_premiersquant}
\frac{1}{\#\{n\in\N^*\colon n\leq x\}}
\#\{n\in\N\cap[1,x]\colon \lambda_f(n)>0\} %
\end{equation}
and
\begin{equation}
\frac{1}{\#\{n\in\N^*\colon n\leq x\}}
\#\{n\in\N\cap[1,x]\colon \lambda_f(n)<0\} %
\end{equation}
are bounded from below. Even though we know by the Sato-Tate Theorem~\cite{MR2827723} that %
\begin{equation}
\lim_{x\to\infty}\frac{1}{\#\{p\in\prem\colon p\leq x\}}
\#\left\{p\in\prem\cap[1,x]\colon \lambda_f(p)>0\right\}=\frac{1}{2} %
\end{equation}
it does not seem easy to deduce a similar limit for~\eqref{eq_premiersquant}.  Lau \& Wu proved also the following result on intervals. There exists \(C>0\) such that, for any \(\varepsilon>0\), there exists \(K>0\) such that for any even integer \(k\geq 4\), for any integer \(N\geq 1\) we have %
\[%
\#\left\{n\in[x,x+CE_Nx^{1/2}]\colon \lambda_f(n)>0\right\}\geq K(Nx)^{1/4-\varepsilon}
\]
as soon as \(x\geq N^2x_0(k)\) where \(x_0(k)\) is a positive real number only depending on \(k\). Here, %
\[%
E_N=N^{1/2}\left(\sum_{d\mid N}d^{-1/2}\log(2d)\right)^3.
\]
An important ingredient used by Lau \& Wu is the following result by Serre~\cite{MR644559}. Let \(f\) be a primitive modular form of weight \(k\) on the congruence subgroup \(\Gamma_0(N)\). Let \(\delta<\dfrac{1}{2}\), there exists \(C>0\) such that, for any \(x\geq 2\), we have %
\[%
\frac{1}{\#\{p\in\prem\colon p\leq x\}}\#\left\{p\in\prem\cap[1,x]\colon \lambda_f(p)=0\right\}\leq\frac{C}{\log(x)^\delta}. %
\]
Such an inequality is missing in the case of Siegel modular forms. %
%
%
\section{Truncated Voronoi formula}

The aim of this section is to establish the following truncated Voronoi formula, which will be needed in the proof of the Theorem.

\begin{lemma}\label{lemmaVoronoi}
Let \(F\) be in \(H_k^*\). Then for any \(A>0\) and \(\varepsilon>0\), we have
\begin{multline}\label{Voronoi}
\sum_{n\leq x} a_F(n) 
 = \frac{x^{3/8}}{(2\pi)^{3/4}} 
\sum_{n\leq M} \frac{a_F(n)}{n^{5/8}} \cos\left(4\sqrt{2\pi} (nx)^{1/4}+\frac{\pi}{4}\right)
\\
+ O_{A, F, \varepsilon}\left((x^3M^{-1})^{1/4+\epsilon} + (xM)^{1/4+\varepsilon}\right)
\end{multline}
uniformly for \(x\geq 2\) and \(1\leq M\leq x^A\), where the implied constant depends on \(A\), \(F\) and \(\varepsilon\) only. In particular %
\begin{equation}\label{MeanValue1/2_deux}
\sum_{n\leq x} a_F(n) \ll_{F, \varepsilon} x^{3/5+\varepsilon}\quad (x\geq 2).
\end{equation}%
\end{lemma}
\begin{proof}
Without loss of generality, we assume that \(M\in\N\). Let \(\kappa:=1+\varepsilon\) and  
\begin{equation}\label{T}
T^4 = 4\pi^2 (M+\tfrac{1}{2})x.
\end{equation}
By the Perron formula (see \cite[Corollary II.2.4]{MR1342300}) 
we have
\begin{equation}\label{3.1}
\sum_{n\leq x} a_F(n)= \frac1{2\pi\ic} \int_{\kappa-\ic T}^{\kappa+\ic T} Z_F(s)\frac{x^s}s\dd s + O_{F, \varepsilon}\left(x^{3/4+\varepsilon}M^{-1/4}+x^\varepsilon\right).
\end{equation}

We shift the line of integration horizontally to \(\re s=-\varepsilon\), the main term gives
\begin{equation}\label{3.3}
\frac1{2\pi \ic } \int_{\kappa-\ic T}^{\kappa+\ic T} Z_F(s)\frac{x^s}s\dd s= Z_F(0) + \frac1{2\pi \ic }\int_{\mathscr{L}} Z_F(s)\frac{x^s}s\dd s,
\end{equation}
where \(\mathscr{L}\) is the contour joining the points \(\kappa\pm \ic T\) and \(-\varepsilon\pm \ic T\).
Using the convexity bound~\cite[\S 1.3]{MR2331346}
\[
Z_F(\sigma+\ic t)\ll_{F, \varepsilon} (\abs{t}+1)^{\max\{2(1-\sigma), \, 0\}+\varepsilon}\quad(-\varepsilon\leq \sigma\leq \kappa),
\] 
the integrals over the horizontal segments and the term \(Z_F(0)\) can be absorbed in \(O_{F, \varepsilon}\left((Tx)^\varepsilon (T+T^{-1}x)\right)=O_{F, \varepsilon}\left(x^{1/4+\varepsilon}M^{1/4}+x^{3/4+\varepsilon}M^{-1/4}\right)\). 

To handle the integral over the vertical segment \(\mathscr{L}_{\mathrm{v}}:=[-\varepsilon-\ic T, -\varepsilon+\ic T]\),  we invoke the functional equation \eqref{FE}. We deduce that 
\begin{equation}\label{3.4}
\frac1{2\pi \ic } \int_{\mathscr{L}_\mathrm{v}} Z_F(s)\frac{x^s}s \dd s= (-1)^k \sum_{n\geq 1}\frac{a_F(n)}n  I_{\mathscr{L}_\mathrm{v}}(nx),
\end{equation}
where   
\[
I_{\mathscr{L}_\mathrm{v}}(y):=\frac1{2\pi \ic } \int_{\mathscr{L}_\mathrm{v}} (2\pi)^{2s-1} \frac{\Gamma(k-\tfrac{1}{2}-s) \Gamma(\tfrac{3}{2}-s)}{\Gamma(s+k-\tfrac{3}{2}) \Gamma(s+\tfrac{1}{2})} \frac{y^s}s \dd s.
\]
By using the Stirling formula
\[
\Gamma(\sigma+\ic t)=\sqrt{2\pi}\abs{t}^{\sigma-1/2}\e^{-\pi\abs{t}/2+\ic (t\log\abs{t}-t)+\ic \sgn(t) (\pi/2) (\sigma-1/2)} \left\{1+O\left(t^{-1}\right)\right\}
\]
uniformly for \(\sigma_1\leq \sigma\leq \sigma_2\) and \(\abs{t}\geq 1\), the quotient of the four gamma factors is 
\begin{equation}\label{quotient}
\abs{t}^{2-4\sigma} \e^{-4\ic (t\log\abs{t}-t)+\ic \sgn(t) \pi (1-k)} \big\{1+O\left(t^{-1}\right)\big\}
\end{equation} 
for bounded \(\sigma\) and any \(\abs{t}\geq 1\), where the implied constant depends on \(\sigma\) and \(k\). Together with the second mean value theorem for integrals \cite[Theorem I.0.3]{MR1342300}, we obtain
\begin{equation}\label{3.5}
\begin{aligned}
I_{\mathscr{L}_\mathrm{v}}(nx)
& \ll (nx)^{-\varepsilon}\left(\abs*{\int_1^T t^{1+4\varepsilon} \e^{-\ic g(t)}\dd t}+ T^{1+4\varepsilon}\right)
\\
& \ll T \left(\frac{T^4}{nx}\right)^\varepsilon\left(\abs*{\int_a^T \e^{-\ic g(t)}\dd t}+1\right)
\end{aligned}
\end{equation}
for some \(1\leq a\leq T\), where \(g(t):=t\log\left(t^4/(4\pi^2nx)\right)-4t\).
In view of \eqref{T}, we have
\[
g'(t)=-\log(4\pi^2nx/t^4)<0
\qquad\text{and}\qquad
\abs{g'(t)}\geq \abs{\log (n/(M+\tfrac12))}
\]
for \(n\geq M+1\) and \(1\leq t\leq T\). Using \eqref{RP-Conjecture} and \cite[Theorem I.6.2]{MR1342300}, we infer that 
\begin{equation}\label{3.6}
\begin{aligned}
\sum_{n>M}\frac{a_F(n)}n  I_{\mathscr{L}_\mathrm{v}}(nx) 
& \ll T \left(\frac{T^4}{x}\right)^\varepsilon 
\sum_{n>M} \frac{d_4(n)}{n^{1+\varepsilon}} 
\left(\abs*{\log \frac{n}{M+\frac12}}^{-1} + 1\right)
\\
& \ll  T \left(\frac{T^4}{x}\right)^\varepsilon  
\left\{\sum_{M<n\leq 2M} \frac{d_4(n) (M+\frac12)}{n^{1+\varepsilon} \abs{n-M-\frac12}}+\frac{1}{M^{\varepsilon/2}}\right\}
\\
& \ll  T \left(\frac{T^4}{\sqrt{M}x}\right)^\varepsilon 
\\ 
& \ll T x^\varepsilon.
\end{aligned}
\end{equation}

For \(n\leq M\), we extend the segment of integration \(\mathscr{L}_\mathrm{v}\) to an infinite line \(\mathscr{L}_\mathrm{v}^*\) in order to apply Lemma~1 in \cite{MR0153643}. Write 
\[
\mathscr{L}_\mathrm{v}^\pm:= [\tfrac{1}{2}+\varepsilon \pm \ic T, \tfrac{1}{2}+\varepsilon \pm \ic \infty),\qquad\mathscr{L}_{\mathrm{h}}^\pm:= [-\varepsilon \pm \ic T, \tfrac{1}{2} + \varepsilon \pm \ic T]
\]
and define \(\mathscr{L}_\mathrm{v}^*\) to be the positively oriented contour consisting of \(\mathscr{L}_\mathrm{v}\), \(\mathscr{L}_\mathrm{v}^\pm\) and \(\mathscr{L}_{\mathrm{h}}^\pm\). In view of \eqref{quotient}, the contribution over the horizontal segments \(\mathscr{L}_{\mathrm{h}}^\pm\) is 
\begin{align*}
I_{\mathscr{L}_{\mathrm{h}}^\pm}(nx) 
& \ll \int_{-\varepsilon}^{1/2-\varepsilon} 
(2\pi)^{2\sigma-1} T^{2-4\sigma} \frac{(nx)^\sigma}{T} \dd\sigma
\\ 
& \ll T \int_{-\varepsilon}^{1/2-\varepsilon} 
\left(\frac{nx}{T^4}\right)^{\sigma} \dd\sigma
\\ 
& \ll T x^\varepsilon.
\end{align*}
As in \eqref{3.5}, for \(n\leq M\) we get that
\begin{align*}
I_{\mathscr{L}_\mathrm{v}^\pm}(nx) 
& \ll (nx)^{1/2+\varepsilon}
\left(\int_{T}^{\infty} t^{-1-4\varepsilon} \e^{-\ic g(t)} \dd t + \frac{1}{T^{1+4\varepsilon}}\right)
\\
& \ll T \left(\frac{nx}{T^4}\right)^{1/2+\varepsilon}
\left(\abs*{\log \frac{M+\frac12}n}^{-1} + 1\right)
\\
& \ll T \left(\abs*{\log \frac{M+\frac12}n}^{-1} + 1\right).
\end{align*}
So
\begin{equation}\label{3.8}
\begin{aligned}
\sum_{n\leq M} \frac{a_F(n)}n  \left(I_{\mathscr{L}_\mathrm{v}^\pm}(nx)+I_{\mathscr{L}_{\mathrm{h}}^\pm}(nx)\right)
& \ll Tx^{\varepsilon/2} \sum_{n\leq M} \frac{d_4(n)}{n}
\left(\abs{\log \frac{M+\frac12}n}^{-1} + 1\right)
\\
& \ll Tx^{\varepsilon/2} \sum_{n\leq M} \frac{d_4(n)(M+\frac12)}{n\abs{n-M-\frac12}}
+ T x^\varepsilon
\\ 
& \ll T x^\varepsilon.
\end{aligned}
\end{equation} 
Define
\[
I_{\mathscr{L}_\mathrm{v}^*}(y)=\frac{1}{4\pi^2 \ic }\int_{\mathscr{L}_\mathrm{v}^*}\frac{\Gamma(k-\tfrac{1}{2}-s) \Gamma(\tfrac{3}{2}-s)\Gamma(s)}{\Gamma(s+k-\tfrac{3}{2}) \Gamma(s+\tfrac{1}{2})\Gamma(1+s)} (4\pi^2 y)^{s} \dd s. 
\]
After a change of variable \(s\) into \(1-s\), we see that 
\[ 
I_{\mathscr{L}_\mathrm{v}^*}(y)=\frac{I_0(4\pi^2 y)}{2\pi},
\]
with 
\[
I_0(t):= \frac1{2\pi \ic }\int_{\mathscr{L}_\varepsilon}\frac{\Gamma(s+k-\tfrac{3}{2}) \Gamma(s+\tfrac{1}{2})\Gamma(1-s)}{\Gamma(k-\tfrac{1}{2}-s) \Gamma(\tfrac{3}{2}-s)\Gamma(2-s)}t^{1-s} \dd s. 
\] 
Here \(\mathscr{L}_\varepsilon\) consists of the line \(s=\tfrac12-\varepsilon+\ic \tau\) with \(\abs{\tau}\geq T\), together with three sides of the rectangle whose vertices are \(\tfrac12-\varepsilon-\ic T\), \(1+\varepsilon-\ic T\), \(1+\varepsilon+\ic T\) and \(\tfrac12-\varepsilon+\ic T\). Note that all the poles of the integrand in \(I_0(t)\) lie on the left of the line \(\mathscr{L}_{\varepsilon}\). 

Using a result due to Chandrasekharan and Narasimhan~\cite[Lemma 1]{MR0153643} generalised by Lau \& Tsang~\cite[Lemma 2.2]{MR1887887} we obtain (note that a factor \(\sqrt{2}\) is missing for the definition of \(e_0\) in both references) %
\[
I_0(t)=\frac{(-1)^k}{\sqrt{2\pi}}t^{3/8}\cos\left(4t^{1/4}+\frac{\pi}{4}\right)+O\left(t^{1/8}\right). %
\]
It hence follows that
\begin{equation}\label{3.9}
I_{\mathscr{L}_\mathrm{v}^*}(nx)=(-1)^k\frac{(nx)^{3/8}}{(2\ic)^{3/4}}\cos\left(4\sqrt{2\pi} (nx)^{1/4}+\frac{\pi}{4}\right) + O\left((nx)^{1/8}\right).
\end{equation}
We conclude 
\begin{equation}\label{3.10}
\sum_{n\leq M}\frac{a_F(n)}n  I_{\mathscr{L}_\mathrm{v}}(nx)=\frac{(-1)^k}{(2\pi)^{3/4}}x^{3/8}\sum_{n\leq M} \frac{a_F(n)}{n^{5/8}} \cos\left(4\sqrt{2\pi} (nx)^{1/4}+\frac{\pi}{4}\right)
+ O\left(x^{1/4+\varepsilon}M^{1/4}\right) 
\end{equation}
from  \eqref{3.8} and \eqref{3.9}. Finally the asymptotic formula \eqref{Voronoi} by \eqref{3.1}-\eqref{3.4}, \eqref{3.6} and \eqref{3.10}.

Since %
\[ %
x^{3/8}\sum_{n\leq M} \frac{a_F(n)}{n^{5/8}} \cos\left(4\sqrt{2\pi} (nx)^{1/4}+\frac{\pi}{4}\right)\ll (xM)^{3/8+\varepsilon}, %
\]
the choice of \(M=x^{3/5}\) in~\eqref{Voronoi} gives~\eqref{MeanValue1/2_deux}.
\end{proof}

\section{Proof of the Theorem}

We establish a lemma that has a similar statement as a one due to Lau \& Wu~\cite[Lemma 3.2]{MR2551607}. However, due to convergence issue, the proof is more delicate.  
%

\begin{lemma}\label{lem1}
Let \(F\) be in \(H_k^*\). Define
\[
S_{F}(x):= \sum_{n\leq x} a_{F}(n).
\]
There exist positive absolute constants \(C, c_1,c_2\) and \(X_0(F)\) depending only on \(F\) such that for all \(X\geq X_0(F)\), we can find \(x_1,x_2\in [X, X+CX^{3/4}]\) for which
\[
S_{F}(x_1)> c_1 X^{3/8}\qquad\text{and}\qquad S_{F}(x_2)< -c_2X^{3/8}. %
\]
\end{lemma}
\begin{proof}
We begin the proof with Theorem~C of Hafner~\cite{MR654524}. 
%
In order to use this result, it is more convenient to introduce the notion of \((C, \ell)\)-summability and to present related simple facts (see~\cite{MR0201863} for more details). Let \(\{g_n(t)\}_{n\geq 0}\) be a sequence of functions.
We write
\[
s(g; n):=\sum_{0\leq\nu\leq n} g_\nu(t),\qquad\sigma(g; n) := \frac{1}{C_n^{(\ell+1)}} \sum_{\nu=0}^n C_{n-\nu}^{(\ell)} s(g; \nu),
\]
where \(C_{n}^{(\ell)} := \binom{\ell+n-1}{n}\). We say that the series of general term \(g_n(t)\) is uniformly \((C, \ell)\)-summable to the sum \(G(t)\) if \(\sigma(g; n)\) converges uniformly to \(G(t)\) as \(n\to\infty\). We have \(C_{0}^{(\ell)}+\cdots + C_{n}^{(\ell)} = C_{n}^{(\ell+1)}\) and if the series \(\sum_n \int g_n(t)\dd t\) converges then the series of general term \(\int g_n(t)\dd t\) is also \((C, \ell)\)-summable and their limits are the same.

As in \cite[page 151]{MR654524}, for \(\rho>-1\) and \(x\notin 2\pi\N\), define
\[
A_\rho(x):=\frac1{\Gamma(\rho+1)} \sum_{2\pi n\leq x} a_{F}(n)(x-2\pi n)^\rho.
\]
Now let \(\mathscr{C}\) be the rectangle with vertices \(c\pm \ic R\) and \(1-b\pm \ic R\) (taken in the counter-clockwise direction), where \(b>c>\max\left\{1, \abs{k-\tfrac{3}{2}}\right\}\) and \(R>\abs*{k-\tfrac{3}{2}}\) are real numbers. Let %
\[ %
Q_\rho(x):= \frac{1}{2\pi\ic } \int_\mathscr{C} \frac{\Gamma(s) (2\pi)^{-s}Z_F(s)}{\Gamma(s+\rho+1)} x^{\rho+s} \dd s. 
\]
Denote by \(\mathscr{C}_{0, b}\) the oriented polygonal path with vertices \(-\ic \infty\), \(-\ic R\), \(b-\ic R\), \(b+\ic R\), \(\ic R\) and \(+\ic \infty\). Let %
\[ %
f_\rho(x) := \frac{1}{2\pi\ic }\int_{\mathscr{C}_{0, b}} \frac{\Gamma(1-s) \Delta(s)}{\Gamma(2+\rho-s) \Delta(1-s)} x^{1+\rho-s} \dd s %
\]
where %
\[
\Delta(s) = \Gamma(s+k-\tfrac{3}{2}) \Gamma(s+\tfrac{1}{2}). %
\]
By \cite[Theorem C]{MR654524}, the series of general term \((-1)^k (2\pi n)^{-1-\rho} a_F(n) f_\rho(2\pi nx)\) is uniformly \((C, \ell)\)-summable for \(\ell>\max\{\tfrac{1}{2}-\rho, 0\}\) on any finite closed interval in \((0, \infty)\) only under the condition \(\rho>-1\) and the sum is \(A_{\rho}(x)-Q_{\rho}(x)\). In particular, we can fix \(\ell=1\) and \(\rho=0\). We shall say \(C\)-summable for \((C,1)\)-summable. 

The only pole of the integrand of \(Q_{0}(x)\) is \(0\), it is encircled by \(\mathscr{C}\) hence %
\begin{equation}\label{UBQrhox}
Q_{0}(x)\ll_F 1\quad(x\geq 1).
\end{equation}

To estimate \(f_{0}(x)\), we use again the result by Lau \& Tsang~\cite[Lemma 2.2]{MR1887887} already used to establish Voronoi formula. We get %
%
%
\begin{equation}\label{7}
f_0(y)=\frac{(-1)^k}{\sqrt{2\pi}}y^{3/8}\cos\left(4y^{1/4}+\frac{\pi}{4}\right)+(-1)^ke_1y^{1/8}\cos\left(4y^{1/4} + \frac{3\pi}{4}\right)+ O\left(\frac{1}{y^{1/8}}\right),
\end{equation}
where \(e_1\) is a absolute constant.

Let
\begin{align*}
\Phi(v)& := (2\pi)^{3/4} \frac{A_{0}(2\pi v^4)}{v^{3/2}},
\\
g_n(v)& := \frac{a_{F}(n)}{n^{5/8}}\cos\left(4\sqrt{2\pi}n^{1/4}v + \frac{\pi}{4}\right),
\\
g_n^*(v)& := \frac{e_1}{v}\frac{a_{F}(n)}{n^{7/8}} \sin\left(4\sqrt{2\pi}n^{1/4}v + \frac{\pi}{4}\right).
\end{align*}
Then the series of general term \(g_n(v)-g_n^*(v)\) is uniformly \(C\)-summable on any finite closed interval in \((0, \infty)\) and the sum is \(\Phi(v) + O(v^{-3/2})\)  (here the term \(O(v^{-3/2})\) comes from \(Q_{0}(2\pi v^4)\) and the \(O\)-term of \eqref{7}). In view of \eqref{MeanValue1/2}, a simple partial integration shows that the series of general term \(g_n^*(v)\) converges to the sum \(\sum_n g_n^*(v)\) uniformly on any finite closed interval in \((0, \infty)\). Thus the series of general term \(g_n(v)\) is uniformly \(C\)-summable on any finite closed interval in \((0, \infty)\) and the sum is \(\Phi(v) + \sum_n g_n^*(v) + O(v^{-3/2})\).

Let \(t\) be any large natural number, \(\kappa>1\) a large parameter that will be fixed later. Write
\[
K_\tau(u) = (1-\abs{u})(1+\tau \cos(4\sqrt{2\pi} \kappa u))
\]
with \(\tau = \pm 1\). We consider the integral
\begin{eqnarray}\label{9}
J_{\tau}= \int_{-1}^1 \Phi(t+\kappa u) K_\tau (u) \dd u.
\end{eqnarray}
We have %
\begin{align*}
\int_{-1}^{1} g_n(t+\kappa u) K_\tau(u) \dd u & = r_{\beta} \frac{a_F(n)}{n^{5/8}},
\\
\int_{-1}^{1} g_n^*(t+\kappa u) K_\tau(u) \dd u & = s_{\beta}e_1 \frac{a_F(n)}{n^{7/8}},
\end{align*}
where
\begin{align*}
r_\beta & := \int_{-1}^1 K_\tau(u)\cos\left(4\sqrt{2\pi}\beta(t+\kappa u) + \frac{\pi}{4}\right) \dd u,
\\
s_\beta & := \int_{-1}^1 \frac{K_\tau(u)}{t+\kappa u}\sin\left(4\sqrt{2\pi}\beta(t+\kappa u) + \frac{\pi}{4}\right) \dd u.
\end{align*}

As in \cite[(3.13)]{MR2551607}, we have
\begin{equation}\label{formularbetarho}
r_\beta= \delta_{\beta =1} \frac{\tau}2+ O\left(\frac{1}{\kappa^2\beta^2}+ \delta_{\beta\neq 1} \frac{1}{\kappa^2(\beta-1)^2}\right)
\end{equation}
and
\begin{equation}\label{UBsbetarho}
s_\beta\ll (t\beta\kappa)^{-1}. %
\end{equation}

It follows that %
\begin{align*}
\int_{-1}^{1} g_1(t+\kappa u) K_\tau(u) \dd u & = \frac{\tau}2+ O\left(\frac{1}{\kappa^2}\right),
\\
\int_{-1}^{1} g_n(t+\kappa u) K_\tau(u) \dd u & \ll \frac{d_4(n)}{\kappa^2 n^{9/8}}\quad(n\geq 2),
\\
\int_{-1}^{1} g_n^*(t+\kappa u) K_\tau(u) \dd u & \ll\frac{d_4(n)}{\kappa tn^{9/8}}, %
\end{align*}
where all the implied constants are absolute. These estimates show that
\begin{align*}
\sum_{n\geq 1} \int_{-1}^{1} g_n(t+\kappa u) K_\tau(u) \dd u& = \frac{\tau}2+ O\left(\frac{1}{\kappa^2}\right),
\\
\sum_{n\geq 1}\int_{-1}^{1} g_n^*(t+\kappa u) K_\tau(u) \dd u & \ll \frac{1}{\kappa t}.
\end{align*}
In view of the remark about \(C\)-summability, we obtain
\begin{equation}\label{11}
J_{\tau}= \frac{\tau}2 + O\left(\frac{1}{\kappa t} + \frac{1}{t^{3/2}}\right).
\end{equation}

We fix \(\kappa\) large enough. When \(X\geq \kappa^4\), we take \(t= \ent*{X^{1/4}}\). So \(t>2 \kappa\) and the \(O\)-term in \(J_\tau\) is \(\ll \kappa^{-2}\), so the main term dominates if \(\kappa\) has been chosen sufficiently large. Therefore
\[
J_{-1}<-\frac{1}{4}\qquad\text{and}\qquad J_1>\frac{1}{4}.
\]
Since \(S_F(x) = A_0(2\pi x)\), we rewrite this as %
\[
\int_{-1}^1\frac{S_F(t+\kappa u)}{(t+\kappa u)^{3/2}}K_{-1}(u)\dd u<-\frac{1}{4(2\pi)^{3/4}}\quad\text{and}\quad\int_{-1}^1\frac{S_F(t+\kappa u)}{(t+\kappa u)^{3/2}}K_{1}(u)\dd u>\frac{1}{4(2\pi)^{3/4}}.
\]

The kernel function \(K_\tau(u)\) is nonnegative and satisfies
\[
1-(3\pi\kappa)^{-2}\leq\int_{-1}^1 K_\tau (u)\dd u \leq2\qquad (\tau = \pm 1).
\]
As a consequence, we have
\[
\frac{S_F((t+\kappa \eta_{+})^4)}{(t+\kappa \eta_{+})^{3/2}}\geq \frac{1}{2(2\pi)^{3/4}}
\]
and
\[
 \frac{S_F((t+\kappa \eta_{-})^4)}{(t+\kappa \eta_{-})^{3/2}}\leq-\frac{1}{4\left(1-(3\pi\kappa)^{-2}\right)(2\pi)^{3/4}}
\]
for some \(\eta_{\pm}\in [-1,1]\). These two points deviate from \(X\) by a distance \(\ll X^{3/4}\), since the difference between \((t\pm \kappa)^4\) is \(\ll \kappa t^3\asymp X^{3/4}\). 

This implies the result of Lemma~\ref{lem1}.
\end{proof}

Now we are ready to prove the Theorem.

By Lemma~\ref{lem1}, for any \(x\geq X_0(F)\) we can pick three points \(x<x_1< x_2<x_3 <x+3Cx^{3/4}\) such that \(S_{F}(x_i)<-cx^{3/8}\) \((i=1,3)\) and \(S_{F}(x_2)> cx^{3/8}\) for some absolute constant \(c>0\). (Note that \(y+Cy^{3/4}\leq x+3Cx^{3/4}\) for \(y=x+Cx^{3/4}\).) Hence we deduce that
\[
\sum_{\substack{x_1<n<x_2\\ a_F(n)>0}} a_F(n)\geq S_{F}(x_2)- S_{F}(x_1)>2cx^{3/8}
\]
and
\[
\sum_{\substack{x_2<n<x_3\\ a_F(n)<0}} \left(-a_F(n)\right)\geq -\left(S_{F}(x_3)- S_{F}(x_2)\right)>2cx^{3/8}.
\]
Thus, the Theorem follows as each term in the two sums are positive and \(\ll_\varepsilon n^\varepsilon\).
\providecommand{\bysame}{\leavevmode\hbox to3em{\hrulefill}\thinspace}
\providecommand{\MR}{\relax\ifhmode\unskip\space\fi MR }
\providecommand{\MRhref}[2]{%
  \href{http://www.ams.org/mathscinet-getitem?mr=#1}{#2}
}
\providecommand{\href}[2]{#2}

\end{document}